\documentclass[12pt]{amsart}
\usepackage{amscd}
\usepackage{amssymb}

\allowdisplaybreaks[1]

\newtheorem{thm}{Theorem}[section]

\newtheorem{lem}[thm]{Lemma}
\newtheorem{prop}[thm]{Proposition}

\newtheorem{defn}[thm]{Definition}

\theoremstyle{definition}

\numberwithin{equation}{section}

\newcommand{\resumename}{R\'esum\'e}

\begin{document}

\date{\today}

\title[Pinczon algebras]{Pinczon algebras}
\author[D. Arnal]{Didier Arnal}
\address{
Institut de Math\'ematiques de Bourgogne\\
UMR CNRS 5584\\
Universit\'e de Bourgogne\\
U.F.R. Sciences et Techniques
B.P. 47870\\
F-21078 Dijon Cedex\\France}
\email{Didier.Arnal@u-bourgogne.fr}

\dedicatory{
\vskip 1cm
\hskip 6cm To my friend Georges:\\
\hskip 5.5cm Forty years ago, we discover together Mathematics...\\
\hskip 4cm And the existence of the trilinear form $I$ for semisimple Lie algebras.\\
\vskip 0.5cm}

\begin{abstract}
After recalling the construction of a graded Lie bracket on the space of cyclic multilinear forms on a vector space $V$, due to Georges Pinczon and Rosane Ushirobira, we prove this construction gives a structure of quadratic associative algebra, up to homotopy, on $V$. In the associative case, it is easy to refind the associated usual Hochshild cohomology. By considering restriction to a subspace or a quotient space of forms, we can present in a completely similar way the cases of quadratic commutative and quadratic Lie algebras, up to homotopy, and the corresponding Harrison and Chevalley cohomologies.\\
\end{abstract}



\maketitle


\section{Introduction}


In recent papers \cite{PU,DPU}, see also \cite{D,MPU}, Georges Pinczon and Rosane Ushirobira introduced what they called a Poisson bracket on the space of forms on a finite dimensional vector space $V$, equipped with a symmetric, non degenerated, bilinear form $b$. If $(e_i)$ is a basis in $V$ and $(e'_j)$ the basis defined by the relations $b(e'_j,e_i)=\delta_{ij}$, the bracket is:
$$
\{\alpha,\beta\}=\sum_i\iota_{e_i}\alpha\wedge\iota_{e'_i}\beta.
$$
Especially, if $\alpha$ is a $(k+1)$-form, and $\beta$ a $(k'+1)$ one, then $\{\alpha,\beta\}$ is a $(k+k')$-form.\\

In fact, the authors proved that a structure of quadratic Lie algebra $(V,[~,~],b)$ on $V$ is completely characterized by a 3-form $I$, such that $\{I,I\}=0$. The relation between the Lie bracket and $I$ is simply:
$$
I(x,y,z)=b([x,y],z),
$$
and the equation $\{I,I\}=0$ is the structure equation. A direct consequence of this construction is the existence of a cohomology on the space of forms, given by:
$$
d\alpha=\{\alpha,I\}.
$$
It is easy to characterize the problem of definition and deformation of quadratic Lie algebra structures on $(V,b)$ with the use of this cohomology on forms.\\

In this paper, we shall just give a generalization of this construction, defining what we call the Pinczon bracket on the space of cyclic multilinear forms on $V$. More precisely, as usual (see \cite{NR}), the structure equation is presented after a shift on degree. Thus, $V$ is now a graded vector space, the space $\mathcal C(V)$ of cyclic multilinear forms on the shifted space $V[1]$ is spanned by the $k$-linear forms (with an usual rule for signs):
$$
(\beta_1\otimes\ldots\otimes\beta_k)^{Cycl}=\sum_j\beta_j\otimes\ldots\beta_k\otimes\beta_1\otimes\ldots\otimes\beta_{j-1},
$$
where the $\beta_j$ are linear forms.

As a slight generalization of \cite{PU}, let us define a Pinczon bracket on $\mathcal C(V)$ as a (graded) skewsymmetric bilinear operator $(\Omega,\Omega')\mapsto\{\Omega,\Omega'\}$ such that, $\{\Omega,\Omega'\}$ is $(k+k')$-linear, if $\Omega$ is $(k+1)$-linear, and $\Omega'$ $(k'+1)$-linear. We moreover suppose that, if $\alpha$ is a linear form, $\{\alpha,~\}$ is a `derivation':
$$
\{\alpha,(\beta_1\otimes\ldots\otimes\beta_k)^{Cycl}\}=\sum_j\{\alpha,\beta_j\}(\beta_1\otimes\ldots\hat{\beta_j}\ldots\otimes\beta_k)^{Cycl},
$$
and defines a (graded) Lie algebra structure on $\mathcal C(V)$, whose center is the space of 0-forms.\\

There is a one-to-one correspondence between the set of Pinczon brackets on $\mathcal C(V)$ and the space of bilinear symmetric forms $b$ on $V$, and the correspondence is given through the Pinczon-Ushirobira formula:
$$
\{\Omega,\Omega'\}=\sum_i\left(\iota_{e_i}\Omega\otimes\iota_{e'_i}\Omega'\right)^{Cycl}.
$$ 

On the other hand, it is well known that the space $\bigotimes^+ V[1]$ is a cogebra for the comultiplication given by the deconcatenation map:
$$
\Delta(x_1\otimes\ldots\otimes x_k)=\sum_{i=1}^{k-1}(x_1\otimes\ldots\otimes x_i)\bigotimes(x_{i+1}\otimes\ldots\otimes x_k).
$$
Moreover, the coderivations of $\Delta$ are characterized by their Taylor series $(Q_k)$ where $Q_k:\bigotimes^kV[1]\longrightarrow V[1]$, and the bracket of such coderivation is still a coderivation (see for instance \cite{BGHHW}).\\

If $B$ is the bilinear form defined on $V[1]$, which corresponds to $b$, there is a bijective map between the cyclic forms $\Omega$ and what we call quadratic coderivation $Q$, given by the formula:
$$
\Omega_Q(x_1,\ldots,x_{k+1})=B(Q(x_1,\ldots,x_k),x_{k+1}).
$$
This relation is an isomorphism of Lie algebras: $\{\Omega_Q,\Omega_{Q'}\}=\Omega_{[Q,Q']}$.\\

With this construction, the notion of quadratic associative algebra, respectively associative quadratic algebra up to homotopy does coincide with the notion of Pinczon algebra structure on $\mathcal C(V)$. This gives also an explicit way to refind the Hochschild cohomology defined by the algebra structure.\\

For any Pinczon bracket, the subspace $\mathcal C_{vsp}(V)$ of cyclic forms, vanishing on shuffle products is a Lie subalgebra. The restriction to this subalgebra of the above construction gives us the notion of quadratic commutative algebra (up to homotopy): it is a Pinczon algebra structure on $\mathcal C_{vsp}(V)$. Similarly, one refinds the Harrison cohomology of commutative algebras.\\

A natural quotient of $\mathcal C(V)$ is the space $\mathcal C(V)|_S$ of totally symmetric multilinear forms on $V[1]$. Any Poisson bracket on $\mathcal C(V)$ induces by symmetrization a bracket on $\mathcal C(V)|_S$, denoted $\{~,~\}|$. This allows us to refind the notion of quadratic Lie algebra (up to homotopy), and the corresponding Chevalley cohomology.\\


\section{Cyclic forms}

\subsection{Koszul's rule}

\

In this paper, we consider a finite dimensional graded vector space $V$, on a charateristic 0 field. Denote $|x|$ the degree of a vector $x$ in $V$. As usual, $V[1]$ is the space $V$ with a shifted degree. If $x$ is homogeneous, its degree in $V[1]$ is
$$
\deg(x)=|x|-1.
$$
Note simply $x$ for $\deg(x)$ when no confusion is possible. Very generally, we use small letters for mapping defined on $V$, and capital letter for `corresponding' mapping, defined on $V[1]$. Let us define now these corresponding mappings.\\

For any real number $a$, define a `$k$-cocycle' $\eta_a$ by putting:
$$
\eta_a(x_1,\ldots,x_k)=(-1)^{\sum_{j\leq k}(a-j)x_j}.
$$
For any permutation $\sigma$ in $\mathfrak S_k$, define $\varepsilon_{|x|}(\sigma)$ as the sign of $\sigma$, taking in account only the positions of the $x_j$ with odd degree $|~|$, and $\varepsilon_{\deg(x)}(\sigma)=\varepsilon_x(\sigma)$ is the sign of the permutation $\sigma$, taking in account only the positions of the $x_j$ with odd degree $\deg$. Let $\varepsilon(\sigma)$ be the sign of $\sigma$. Then a direct computation shows (see for instance \cite{AAC1}):
$$
\eta_a(x_{\sigma(1)},\ldots,x_{\sigma(k)})\eta_a(x_1,\ldots,x_k)=\varepsilon(\sigma)\varepsilon_{|x|}(\sigma)\varepsilon_x(\sigma).
$$

If $q$ is a $k$-linear mapping from $V^k$ into a graded vector space $W$, define the associated $k$-linear mapping from $V[1]^k$ into $W[1]$, with $\deg(Q)=|q|+k-1$, by
$$
Q(x_1,\ldots,x_k)=\eta_k(x_1,\ldots,x_k)q(x_1,\ldots,x_k).
$$

If $W$ is the base field $\mathbb K$, we generally prefer to keep the 0 degree for scalars. If $\alpha$ is a $k$-linear form on $V$, we associate to $\alpha$ a form $A$, on $V[1]$, with degree $\deg(A)=|A|+k$.\\ 

Let us now recall the sign rule due to Koszul (\cite{K}). For any relation using letters representing graded objects, in which the ordering of the letters is modified in some terms, it is always understood that, in front of the first term, there is a + sign and in front of every other term, there is an (implicit) sign $\varepsilon_{letters}(\sigma)$ where $\sigma$ is the permutation of the letters between the first term and the other.\\

For instance, the relation:
$$
F(A(x,y),z)=A(F(x,z),y)+F(x,A(y,z))
$$
means in an explicit form:
$$
F(A(x,y),z)=(-1)^{AF+yz}A(F(x,z),y)+(-1)^{Ax}F(x,A(y,z)).
$$

\subsection{Cyclic product.}

\

\begin{defn}

\

If $\Omega$ is a $(k+1)$-linear form, let us say that $\Omega$ is a {\it cyclic} form on $V[1]$, if it satisfies (with the above rule):
$$
\Omega(x_{k+1},x_1,\ldots,x_k)=\Omega(x_1,\ldots,x_{k+1}).
$$

Denote $\mathcal C(V)$ the space of cyclic forms on $V[1]$.\\
\end{defn}

For any $k+1$-linear form $\Omega$, define the action of $\sigma$ in $\mathfrak S_{k+1}$ on $\Omega$ as:
$$
(\Omega^\tau)(x_1,\ldots,x_{k+1})=\Omega(x_{\tau^{-1}(1)},\ldots,x_{\tau^{-1}(k+1)}).
$$
Then the cyclic form $\Omega^{Cycl}$ is as follows: if $Cycl$ is the subgroup of $\mathfrak S_{k+1}$ generated by the cycle $(1,2,\ldots,k+1)$,
$$
\Omega^{Cycl}=\sum_{\tau\in Cycl}\Omega^\tau.
$$

Thanks to this operator, it is easy to `cyclicize' the tensor product of multilinear forms:\\

\begin{defn}

\

We define the cyclic product of multilinear forms $A$, $B$ on $V[1]$ as:
$$
A\odot B=(A\otimes B)^{Cycl}.
$$
\end{defn}

\begin{lem}

\

The cyclic product is (graded) commutative, but non associative.\\
\end{lem}

\begin{proof}

Suppose $A$ is a $p$-linear cyclic form and $B$ a $q$-linear one. Then, $A\odot B$ is a cyclic $p+q$-form. Consider the permutation:
$$
\sigma=\left(\begin{matrix}1&\ldots&p&p+1&\ldots&p+q\\
q+1&\ldots&p+q&1&\ldots&q\end{matrix}\right),
$$
then $\sigma\in Cycl$ and $B\otimes A=(A\otimes B)^\sigma$. Therefore
$$
B\odot A=\sum_{\tau\in Cycl}(B\otimes A)^\tau=\sum_{\tau\in Cycl}(A\otimes B)^\tau=A\odot B.
$$

To see the product is non associative, suppose $V$ is 6 dimensional, with a basis $(e_i)$. Denote $\epsilon_i$ the dual basis, and consider the symmetric bilinear forms:
$$
A=\epsilon_1\odot \epsilon_2,\quad B=\epsilon_3\odot \epsilon_5,\quad \Gamma=\epsilon_6\odot \epsilon_4.
$$
Then a direct computation gives
$$
\left((A\odot B)\odot \Gamma\right)(e_1,\ldots,e_6)=1,\quad\text{and}\quad\left(A\odot (B\odot \Gamma)\right)(e_1,\ldots,e_6)=0.
$$
\end{proof}

\subsection{Pinczon bracket}

\

\begin{defn}

\

A Pinczon bracket on the space $\mathcal C(V)$ of cyclic multilinear forms on a graded space $V[1]$ is a bilinear map:
$$
(\Omega,\Omega')\mapsto\{\Omega,\Omega'\}
$$
such that
\begin{itemize}
\item[1.] If $\mathcal C^k$ is the space of $k$-linear cyclic forms, $\{\mathcal C^{k+1},\mathcal C^{k'+1}\}\subset\mathcal C^{k+k'}$,\\
\item[2.] $\mathcal C(V)$, equipped with $\{~,~\}$ is a graded Lie algebra: $\{\Omega,\Omega'\}=-\{\Omega',\Omega\}$, and
$$
\{\Omega,\{\Omega',\Omega''\}\}+\{\Omega',\{\Omega'',\Omega\}\}+\{\Omega'',\{\Omega,\Omega'\}\}=0.
$$
\item[3.] for any linear form $\alpha$, $\{\alpha,~\}$ is like a derivation: for any $\beta_1,\ldots,\beta_k$ linear,
$$
\{\alpha,(\beta_1\otimes\ldots\otimes\beta_k)^{Cycl}\}=\sum_j\{\alpha,\beta_j\}(\beta_1\otimes\ldots\hat{\beta_j}\ldots\otimes\beta_k)^{Cycl},
$$
\item[4.] $\{\Omega,\mathcal C(V)\}=0$ if and only if $\Omega$ is in $\mathcal C^0$.\\
\end{itemize}
\end{defn}

The space of such brackets is in fact in one-to-one correspondence with the space of symmetric, non degenerated bilinear forms $b$ on $V$ (see \cite{PU}).\\

\begin{thm}

\

\begin{itemize}
\item[1.] There is a bijective map between the space $\mathcal P$ of Pinczon bracket on $\mathcal C(V)$ and the space $\mathcal B$ of degree 0, symmetric, non degenerated bilinear forms $b$ on $V$,
\item[2.] Let $b$ be an element in $\mathcal B$, $(e_i)$ be any basis for $V$, if $(e'_i)$ is the basis defined by $b(e_i,e'_j)=\delta_{ij}$, then the Pinczon bracket associated to $b$ is:
$$
\{\Omega,\Omega'\}=\sum_i\iota_{e_i}\Omega\odot\iota_{e'_i}\Omega'.
$$
\end{itemize}
\end{thm}

Remark the last formula is exactly the Pinczon-Ushirobira one, given in \cite{PU}. The sign in front of the bracket in \cite{PU} is the sign $\eta_2$. 

\begin{proof}

Let $\{~,~\}$ be a Pinczon bracket on $\mathcal C(V)$, then:
$$
\{\mathcal C^0,\mathcal C(V)\}=0,\quad\text{ and }\quad\{\mathcal C^1,\mathcal C^1\}\subset\mathcal C^0=\mathbb K.
$$
Thus the bracket defines a degree 0 bilinear, antisymmetric form $B^\star$ on $\mathcal C^1=(V[1])^\star$. As above this defines a symmetric bilinear form $b^\star$ on the space $(V[1])^\star[1]$, which is canonically isomorphic to $V^\star$.

Let us suppose $b^\star$ (or $B^\star$) degenerated, then there is $\alpha$ in $\mathcal C^1$ such that
$$
B^\star(\alpha,\mathcal C^1)=\{\alpha,\mathcal C^1\}=0.
$$
Then, for any linear forms $\beta_1,\ldots,\beta_k$, $\{\alpha,(\beta_1\otimes\ldots\otimes\beta_k)^{Cycl}\}=0$, or $\{\alpha,\mathcal C^k\}=0$, and $\alpha$ is a central element. This is impossible, $b^\star$ is non degenerated and allows to identify $V^\star$ and $V$ and to define a non degenerated bilinear symmetric form $b$ on $V$.\\

Let $B^\star$ be a skewsymmetric bilinear form on $(V[1])^\star$. For any basis $(e_i)$ of $V[1]$ there are vectors $e'_i$ such that (with the rule on signs):
$$
B^\star(\alpha,\beta)=\sum_i\iota_{e_i}\alpha\otimes \iota_{e'_i}\beta=\sum_i\alpha(e_i)\beta(e'_i)=\sum_i\iota_{e_i}\alpha\odot\iota_{e'_i}\beta.
$$

Coming back to $V^\star$ this means, if all the objects are homogeneous,
$$\aligned
b^\star(\alpha,\beta)&=\sum_i(-1)^{|\beta|}\alpha(e_i)\beta(e'_i).
\endaligned
$$
Identify $V^\star$ to $V$ by defining, for any $\gamma$ in $V^\star$, the vector $x_\gamma$ in $V$ such that $\alpha(x_\gamma)=b^\star(\alpha,\gamma)$, for any $\alpha$ in $V^\star$. Then, if $(\epsilon_i)$ is the dual basis of $(e_i)$ ($\epsilon_i(e_j)=\delta_{ij}$),
$$
b^\star(\epsilon_i,\beta)=(-1)^{|\beta||\epsilon_i|}\beta(x_{\epsilon_i})=(-1)^{|\beta|}\beta(e'_i).
$$
Therefore $e'_i=x_{\epsilon_i}$, $(e'_i)$ is a basis for $V$, and with the same computation, if $(\epsilon'_i)$ is the dual basis, $x_{\epsilon'_i}=(-1)^{|e_i|}e_i$, and $b(e'_j,e_i)=\delta_{ij}$.

Consider now the bracket:
$$
\{\Omega,\Omega'\}_{PU}=\sum_i\iota_{e_i}\Omega\odot \iota_{e'_i}\Omega',
$$
where $\Omega$ is a $(k+1)$-linear cyclic form and $\Omega'$ a $(k'+1)$ one. In a following section, we shall prove this bracket defines a graded Lie algebra structure on $\mathcal C(V)$. It is clear that, for any $\alpha$ and $\beta_j$ linear,
$$
\{\alpha,(\beta_1\otimes\ldots\otimes\beta_k)^{Cycl}\}_{PU}=\sum_j\{\alpha,\beta_j\}_{PU}(\beta_1\otimes\ldots\hat{\beta_j}\ldots\otimes\beta_k)^{Cycl}.
$$
Moreover the center of the Lie algebra $(\mathcal C(V),\{~,~\}_{PU})$ is $\mathcal C^0$. In other word, $\{~,~\}_{PU}$ is a Pinczon bracket.\\

If $k+k'\leq0$, $\{\Omega,\Omega'\}=\{\Omega,\Omega'\}_{PU}$. Suppose by induction this relation holds for $k+k'<N$ and consider $\Omega$ and $\Omega'$ such that $k+k'=N$. For any $i$,
$$\aligned
\{\epsilon_i,\{\Omega,\Omega'\}\}&=-\{\Omega,\{\Omega',\epsilon_i\}\}-\{\Omega',\{\epsilon_i,\Omega\}\}\\
&=-\{\Omega,\{\Omega',\epsilon_i\}_{PU}\}_{PU}-\{\Omega',\{\epsilon_i,\Omega\}_{PU}\}_{PU}\\
&=\{\epsilon_i,\{\Omega,\Omega'\}_{PU}\}_{PU}
=\iota_{e'_i}\{\Omega,\Omega'\}_{PU}
\endaligned
$$

On the other hand, if $\{\Omega,\Omega'\}=\sum_\beta(\beta_1\otimes\ldots\otimes\beta_{k+k'})^{Cycl}$,
$$\aligned
\{\epsilon_i,\{\Omega,\Omega'\}\}&=\sum_{\beta,j}\beta_j(e'_i)(\beta_1\otimes\ldots\hat{\beta_j}\ldots\otimes\beta_{k+k'})^{Cycl}\\
&=\sum_\beta \iota_{e'_i}(\beta_1\otimes\ldots\otimes\beta_{k+k'})^{Cycl}
=\iota_{e'_i}\{\Omega,\Omega'\}.
\endaligned
$$
Since $(e'_i)$ is a basis for $V[1]$, this implies $\{\Omega,\Omega'\}_{PU}=\{\Omega,\Omega'\}$, and proves existence and unicity of the Pinczon bracket associated to the symmetric, non degenerated bilinear form $b$ on $V$ and the expression of this bracket.\\

\end{proof}



\section{Codifferential}


\

\subsection{General construction}

\

The tensor algebra $\bigotimes^+V[1]=\sum_{k>0}\bigotimes^kV[1]$ has a natural comultiplication: the deconcatenation map $\Delta$:
$$
\Delta(x_1\otimes\ldots\otimes x_k)=\sum_{r=1}^{k-1}(x_1\otimes\ldots\otimes x_r)\bigotimes(x_{r+1}\otimes\ldots\otimes x_k).
$$

It is well known (see for instance \cite{AAC2,BGHHW}) that any multilinear mapping $Q$ can be extended in an unique way into a coderivation $D_Q$ of $\Delta$, moreover to any coderivation $D$ corresponds a sequence $(Q_k)_{k\geq1}$ of $k$-linear mappings such that $D=\sum_k D_{Q_k}$ (this sequence can be viewed as the Taylor series of $D$). Moreover, the space of coderivations is a natural graded Lie algebra, thus the space of multilinear mappings is a graded Lie algebra for the bracket:
$$
[Q,Q']=Q\circ Q'-Q'\circ Q,
$$
with, if $x_{[a,b]}$ denotes the tensor product $x_a\otimes x_{a+1}\otimes\dots\otimes x_b$,
$$
Q\circ Q'(x_{[1,k+k'-1]})=\sum_{r=0}^{k-1}Q(x_{[1,r]},Q'(x_{[r+1,r+k']}),x_{[r+k'+1,k+k'-1]}).
$$

\begin{lem}

The space $\mathcal D$ of multilinear mappings is a Lie algebra for its bracket.\\
\end{lem}

\begin{proof}
This is well known (\cite{NR}), let us give here a sketch of the proof. Indeed:
$$\aligned
(Q\circ Q')\circ Q''(x_{[1,k+k'+k''-2]})&=\\
&\hskip -3cm=\sum Q(x_{[1,s]},Q''(x_{]s,s+k'']}),x_{]s+k'',r]},Q'(x_{]r,r+k']}),x_{]r+k',k+k'+k''-2]})\\
&\hskip-2cm +\sum Q(x_{[1,r]},Q'(x_{]r,s]},Q''(x_{]s,s+k'']}),x_{]s+k'',r+k'-1]}),x_{]r+k'-1,k+k'-2]})\\
&\hskip -2cm+\sum Q(x_{[1,r]},Q'(x_{]r,r+k']}),x_{]r+k',s]},Q''(x_{]s,s+k'']}),x_{]s+k'',k+k'+k''-2]}).
\endaligned
$$
This relation implies:
$$
(Q\circ Q')\circ Q''-(Q\circ Q'')\circ Q=Q\circ (Q'\circ Q''-Q''\circ Q'),
$$
and the Jacobi identity:
$$
[[Q,Q'],Q'']+[[Q',Q''],Q]+[[Q'',Q],Q']=0.
$$
\end{proof}

\subsection{Relation with the Pinczon bracket}

\

Consider a vector space $V$ equipped with a symmetric, non degenerated bilinear form $b$, with degree 0, define the symplectic form $B=\eta_2b$ on $V[1]$:
$$
B(x,y)=(-1)^xb(x,y).
$$ 
For any $k$ linear map $Q$ from $V[1]^k$ into $V[1]$, define the $(k+1)$-linear form:
$$
\Omega_Q(x_1,\ldots,x_{k+1})=B(Q(x_1,\ldots,x_k),x_{k+1}),
$$
and let us say that $Q$ is $B$-quadratic if and only if $\Omega_Q$ is cyclic. Denote $\mathcal D_B$ the space of $B$-quadratic multilinear maps (or coderivations).\\

The fundamental examples of cyclic maps are mappings associated to a Lie bracket or an associative multiplication on $V$. More precisely, if $(x,y)\mapsto q(x,y)$ is any internal law, with degree 0, and $Q(x,y)=(-1)^xq(x,y)$, then $\Omega_Q$ is cyclic if and only if:
$$\aligned
b(q(x,y),z)&=(-1)^{x+z}\Omega_Q(x,y,z)\\
&=(-1)^{x+z+x(y+z)}\omega_Q(y,z,x)\\
&=(-1)^{x+z+x(y+z)+y+3x}b(q(y,z),x)\\
&=(-1)^{(x+1)(y+z)}b(q(y,z),x)\\
&=b(x,q(y,z)),
\endaligned
$$
if and only if $(V,q,b)$ is an algebra with an invariant bilinear form, {\it i.e.} a quadratic algebra.\\

Remark that if $q$ is a Lie bracket (skewsymmetric), then $\omega_q$ is (totally) skewsymmetric, $\Omega_Q$ is (totally) symmetric.

Now, if $q$ is a commutative (and associative) product, and $b$ is invariant, $\omega_q$ is vanishing on the image of the twisted shuffle product on the 2 first variables. In fact there is an unique such product, defined by 
$$
(x_1\otimes x_2)\mapsto(x_1\otimes x_2)-(x_2\otimes x_1).
$$
$\Omega_Q$ is vanishing on the image of the shuffle product on the 2 first variables (with the sign rule):
$$
sh_{(1,1)}(x_1\otimes x_2)=(x_1\otimes x_2)+(x_2\otimes x_1).
$$

\begin{prop}

The space $\mathcal D_B$ of $B$-quadratic maps $Q$ is a Lie subalgebra of $\mathcal D$.

The space $\mathcal C(V)$ of cyclic forms, equipped with the Pinczon bracket associated to $b$ is a graded Lie algebra, isomorphic to $\mathcal D_B$. \\
\end{prop}

\begin{proof}
For any sequence $I=\{i_1,\ldots,i_k\}$ of indices, denote $x_I$ the tensor $x_{i_1}\otimes\ldots\otimes x_{i_k}$.
 
Remark that (with our notations), if $Q$ is $k$-linear and $Q'$ $k'$-linear, without the sign rule, but with explicit signs,
$$\aligned
A&=(-1)^{Q'+\sum_{j>k}x_j}b(Q(x_{[1,k]}),Q'(x_{]k,k+k']}))\\
&=(-1)^{Q'+\sum_{j>k}x_j+(Q'+\sum_{j>k}x_j+1)(Q+\sum_{i\leq k}x_i+1)}b(Q'(x_{]k,k+k']}),Q(x_{[1,k]}))\\
&=-(-1)^{Q+\sum_{i\leq k}x_i}(-1)^{(Q'+\sum_{j>k}x_j)(Q+\sum_{i\leq k}x_i)}b(Q'(x_{]k,k+k']}),Q(x_{[1,k]})),
\endaligned
$$
Suppose $Q$ and $Q'$ $B$-quadratic, thus, with the Koszul rule, this is just:
$$\aligned
B(Q(x_{[1,k]}),Q'(x_{]k,k+k']}))&=-B(Q'(Q(x_{[1,k]}),x_{]k,k+k'-1]}),x_{k+k'})\\
&=B(Q(Q'(x_{]k,k+k']}),x_{[1,k-1]}),x_k).
\endaligned
$$
Therefore:
$$\aligned
&B(Q\circ Q'(x_{[2,k+k']}),x_1)=\sum_{1\leq r}B(Q(x_{[2,r]},Q'(x_{]r,r+k']}),x_{]r+k',k+k']}),x_1)\\
&=\sum_{1\leq r<k}B(Q(x_{[1,r]},Q'(x_{]r,r+k']}),x_{]r+k',k+k'-1]}),x_{k+k'})
+B(Q(x_{[1,k]}),Q'(x_{]k,k+k']}))\\
&=\sum_{1\leq r<k}B(Q(x_{[1,r]},Q'(x_{]r,r+k']}),x_{]r+k',k+k'-1]}),x_{k+k'})
-B(Q'(Q(x_{[1,k]}),x_{]k,k+k'-1]}),x_{k+k'})\\
\endaligned
$$
Or
$$\aligned
&B([Q,Q'](x_{[2,k+k']}),x_1)=\\
&=\sum_{1\leq r<k}B(Q(x_{[1,r]},Q'(x_{]r,r+k']}),x_{]r+k',k+k'-1]}),x_{k+k'})
+B(Q(Q'(x_{[1,k']}),x_{]k',k+k'-1]}),x_{k+k'})\\
&
-\sum_{1\leq r<k'}B(Q'(x_{[1,r]},Q(x_{]r,r+k]}),x_{]r+k,k+k'-1]}),x_{k+k'})
-B(Q'(Q(x_{[1,k]}),x_{]k,k+k'-1]}),x_{k+k'})\\
&=B([Q,Q'](x_{[1,k+k'-1]}),x_{k+k'})
\endaligned
$$

Thus $[Q,Q']$ is $B$-quadratic, $\mathcal D_B$ is a Lie subalgebra of $\mathcal D$.\\

On the other hand, $\Omega_Q$ (resp. $\Omega_{Q'}$) is a $k+1$-linear (resp. $k'+1$-linear) cyclic form, and, with a little abuse of notations,
$$\aligned
\{\Omega_Q,\Omega_{Q'}\}&(x_1,\ldots,x_{k+k'})=\left(\sum_i\iota_{e_i}\Omega_Q\otimes\iota_{e'_i}\Omega_{Q'}\right)^{Cycl}(x_1,\ldots,x_{k+k'})\\
&=\sum_{\sigma\in Cycl}\sum_i B(Q(x_{\sigma^{-1}(1)},\ldots,x_{\sigma^{-1}(k)}),e_i)
B(Q'(x_{\sigma^{-1}(k+1)},\ldots,x_{\sigma^{-1}(k+k')}),e'_i)\\
&=\sum_{\sigma\in Cycl}B(Q(x_{\sigma^{-1}([1,k])}),Q'(x_{\sigma^{-1}([k+1,k+k'])})).
\endaligned
$$

Consider a term in this sum, such that $k+k'$ belongs to $\sigma^{-1}([1,k])$. This term is:
$$
B(Q(x_{[r+k'+1,k+k']},x_{[1,r]}),Q'(x_{[r+1,r+k']}))=B(Q(x_{[1,r]},Q'(x_{[r+1,r+k']}),x_{[r+k'+1,k+k'-1]}),x_{k+k'}),
$$
and the sum of all these terms is just $B((Q\circ Q')(x_{[1,k+k'-1]},x_{k+k'}))$.

Similarly, a term such that $k+k'$ is in $\sigma^{-1}([k+1,k+k'])$ is:
$$
B(Q(x_{[r+1,r+k']}),Q'(x_{[r+k+1,k+k']},x_{[1,r]}))=-B(Q'(x_{[1,r]},Q(x_{[r+1,r+k]}),x_{[r+k+1,k+k'-1]}),x_{k+k'}),
$$
and the corresponding sum is $-B((Q'\circ Q)(x_{[1,k+k'-1]},x_{k+k'}))$.

This proves:
$$
\{\Omega_Q,\Omega_{Q'}\}=\Omega_{[Q,Q']}.
$$
Since the map $Q\mapsto\Omega_Q$ is bijective, this proves the proposition.\\
\end{proof}

Let us now study in a more detailled way the three main cases, when $Q$ corresponds to an associative, or a commutative, or a Lie strucuture, or to such a structure up to homotopy.\\


\section{Associative Pinczon algebras}


\

\subsection{Associative quadratic algebras}

\

Suppose now $q$ is an associative law (with degree 0) on $V$, and $b$ is invariant, then $q$ defines a coderivation $Q$ of $\Delta$ on $\otimes^+V[1]$, with degree 1, which is the Bar resolution of the associative algebra $(V,q)$. The associativity of $q$ is equivalent to the relation $[Q,Q]=0$.

More generally, a structure of $A_\infty$ algebra (or associative algebra up to homotopy) on the space $V$ is a degree 1 coderivation $Q$ of $\Delta$ on $\otimes^+V[1]$, such that $[Q,Q]=0$. With this last relation, the map
$$
\Lambda\mapsto\{\Omega_Q,\Lambda\}
$$
is a degree 1 differential on the (graded) Lie algebra $\mathcal C(V)$. The corresponding cohomo\-logy is the Pinczon cohomology of multilinear forms.\\

\begin{defn}

\

An associative Pinczon algebra $(\mathcal C(V),\{~,~\},\Omega)$ is a vector space $V$, such that $\mathcal C(V)$ is equipped with a Pinczon bracket $\{~,~\}$, and an element $\Omega$ in $\mathcal C(V)$, with degree 3, such that $\{\Omega,\Omega\}=0$.\\
\end{defn}

If $\Omega$ is trilinear, then an associative Pinczon algebra is simply a quadratic associative algebra $(V,b,q)$, where $b$ is the symmetric non degenerated form coming from the restriction of the Pinczon bracket to $\left((V[1])^\star\right)^2$, and $q$ is the bilinear mapping associated to $Q$ such that $\Omega=\Omega_Q$.\\

\begin{prop}

\

Let $(\mathcal C(V),\{~,~\},\Omega)$ be an associative Pinczon algebra, then there is an unique symplectic form $B$ on $V[1]$ and an unique $B$-quadratic coderivation $Q$ of $(\bigotimes^+(V[1]),\Delta)$ such that $\Omega=\Omega_Q$.

Conversely, if $(V,b)$ is a vector space equipped with a non degenerated symmetric bilinear form $b$, any $B$-quadratic structure $Q$ of $A_\infty$ algebra defines an unique associative Pinczon algebra $(\mathcal C(V),\{~,~\},\Omega_Q)$.\\
\end{prop}

In other words, any associative Pinczon algebra $(\mathcal C(V),\{~,~\},\Omega)$ is equivalent to a structure of quadratic associative algebra up to homotopy on $V$.\\

\subsection{Bimodules and Hochschild cohomology}

\

In this situation, it is possible to recover the usual Hochschild cohomology of any (graded) bimodule $M$ for an associative algebra $(V,q)$.\\

First, put $q(x,y)=xy$ and consider the semidirect product $W=V\rtimes M$, that is the vector space $V\times M$, equipped with the associative multiplication:
$$
q_W((x,a),(y,b))=(xy,x\cdot b+a\cdot y).
$$ 
Now, since $W$ is a $(W,q_W)$ bimodule, its dual $W^\star=V^\star\times M^\star$ is also a $(W,q_W)$-bimodule, with:
$$
\big((x,a)\cdot f\big)(z,c)=f((z,c)(x,a)),\qquad \big(f\cdot(x,a)\big)(z,c)=f((x,a)(z,c)),
$$
or if $f=(g,h)\in V^\star\times M^\star$,
$$\aligned
(x,a)\cdot(g,h)(z,c)&=g(zx)+h(z\cdot a)+h(c\cdot x)=(x\cdot g+a\cdot h,x\cdot h)(z,c),\\
(g,h)\cdot(x,a)(z,c)&=g(xz)+h(a\cdot z)+h(x\cdot c)=(g\cdot x+h\cdot a,h\cdot x)(z,c).
\endaligned
$$

These objects define a structure of associative algebra on the space $\tilde{V}=W\times W^\star$, namely:
$$
\tilde{q}((x,a,g,h),(x',a',g',h'))=(xx',x\cdot a'+a\cdot x',x\cdot g'+g\cdot x'+a\cdot h'+h\cdot a',x\cdot h'+h\cdot x'),
$$
and a non degenerated symmetric bilinear form $\tilde b$,
$$
\tilde{b}((x,a,g,h),(x',a',g',h'))=g(x')+h(a')+g'(x)+h'(a).
$$
A direct computation shows that $(\tilde{V},\tilde{b},\tilde{q})$ is a quadratic associative algebra, it is the usual notion of double extension of $\{0\}$ by $W$.\\

\begin{defn}

\

Let $(V,q)$ be an associative algebra and $M$ a bimodule. The double semi-direct product of $V$ by $M$ is the quadratic associative algebra $(\tilde{V},\tilde{b},\tilde{q})$.

The corresponding associative Pinczon algebra $(\mathcal C(\tilde{V}),\{~,~\}\tilde{~},\Omega_{\tilde{Q}})$ is the Pinczon double semi-direct product of $V$ by $M$.\\
\end{defn}

Look now for a $k$-linear mapping $c$ from $V^k$ into $M$, with degree $|c|=2-k$, and consider it as a mapping $C'$, with degree 1, from $\tilde{V}[1]^k$ into $\tilde{V}[1]$, as follows ($x_j\in V$, $a_j\in M$, $g_j\in V^\star$, $h_j\in M^\star$):
$$
C'((x_1,a_1,g_1,h_1),\ldots,(x_k,a_k,g_k,h_k))=(0,C(x_1,\ldots,x_k),0,0).
$$
But the form $\Omega_{C'}$ being not cyclic, $C'$ is not a quadratic map. Therefore replace $C'$ by $\tilde{C}$ such that $\Omega_{\tilde{C}}=(\Omega_{C'})^{Cycl}$. Precisely:
$$
\Omega_{C'}((x_1,a_1,g_1,h_1),\ldots,(x_{k+1},a_{k+1},f_{k+1},g_{k+1}))=h_{k+1}(C(x_1,\ldots,x_k)),
$$
and
$$
\tilde{C}((x_1,a_1,g_1,h_1),\ldots,(x_k,a_k,g_k,h_k))=(0,C(x_1,\ldots,x_k),\sum_{j=1}^kC_j(x_1,\ldots,h_j,\ldots,x_k),0),
$$
with
$$
C_j(x_1,\ldots,h_j,\ldots,x_k)(y)=h_j(C(x_{j+1},\ldots,x_k,y,x_1,\ldots,x_{j-1})).
$$

Then, if $\tilde{Q}$ is associated to $\tilde{q}$ as above, denote simply:
$$\aligned
Q(x,a)&=\tilde{Q}((x,0,0,0),(0,a,0,0))=(-1)^xx\cdot a,\\
Q(a,x)&=\tilde{Q}((0,a,0,0),(x,0,0,0))=(-1)^aa\cdot x,
\endaligned
$$
and so on. Recall that, for instance, for any $h\in M^\star[1]$, $a\in M[1]$, and $y\in M[1]$, $Q(a,h)$ and $Q(h,a)$ are in $V^\star[1]$, and:
$$\aligned
Q(x,g)(y)&=h(Q(y,x)),&\hskip 4cm~Q(g,x)(y)&=h(Q(x,y)),\\
Q(a,h)(y)&=h(Q(y,a)),&\hskip 4cm~Q(h,a)(y)&=h(Q(a,y)).
\endaligned
$$
Let us now compute $[\tilde{Q},\tilde{C}]=(0,[Q,C],\sum_j[Q,C_j],0)$. Let $D_j$ be the term depending linearly in $h_j$. Consider the cases $j=1$, $1<j<k+1$, $j=k+1$.\\

\noindent
\underline{Case: $j=1$}, then:
$$\aligned
D_1&=Q(h_1,C(x_2,\ldots,x_{k+1}))+Q(C_1(h_1,x_2,\ldots,x_k),x_{k+1})+C_1(Q(h_1,x_2),x_3,\ldots,x_{k+1})+\\
&\hskip 1cm+\sum_{r=2}^kC_1(h_1,x_2,\ldots,Q(x_r,x_{r+1}),\ldots,x_{k+1}).
\endaligned
$$
The evaluation of this form on $y\in V[1]$ is:
$$\aligned
D_1(y)&=h_1\Big[Q(C(x_2,\ldots,x_{k+1}),y)+C(x_2,\ldots,x_k,Q(x_{k+1},y))+Q(C(x_2,\ldots,x_{k+1}),y)+\\
&\hskip 1cm+\sum_{r=2}^kC(x_2,\ldots,Q(x_r,x_{r+1}),\ldots,x_{k+1},y)\Big]\\
&=h_1\left([Q,C](x_2,\ldots,x_{k+1},y)\right)=\left([Q,C]_1(h_1,\ldots,x_{k+1})\right)(y).
\endaligned
$$

\

\noindent
\underline{Case: $1<j<k+1$}, then:
$$\aligned
D_j&=Q(x_1,C_j(x_2,\ldots,h_j,\ldots,x_{k+1}))+Q(C_j(x_1,\ldots,h_j,\ldots,x_k),x_{k+1})+\\
&+C_{j-1}(x_1,\ldots,Q(x_{j-1},h_j),\ldots,x_{k+1})+C_j(x_1,\ldots,Q(h_j,x_{j+1}),\ldots,x_{k+1})+\\
&+\sum_{r<j-1}C_{j-1}(x_1,\ldots,Q(x_r,x_{r+1}),\ldots,h_j,\ldots,x_{k+1})+\\
&+\sum_{r>j+1}C_j(x_1,\ldots,h_j,\ldots,Q(x_r,x_{r+1}),\ldots,x_{k+1}).
\endaligned
$$
Evaluating this form on $y\in V[1]$, one gets:
$$\aligned
&D_j(y)=\\
&h_j\Big[C(x_{j+1},\ldots,x_{k+1},Q(y,x_1),x_2,\ldots,x_{j-1})+C(x_{j+1},\ldots,Q(x_{k+1},y),x_1,x_2,\ldots,x_{j-1})+\\
&+Q(C(x_{j+1},\ldots,x_{k+1},y,x_1,\ldots,x_{j-2}),x_{j-1})+Q(x_{j+1},C(x_{j+2},\ldots,x_{k+1},y,x_1,\ldots,x_{j-1}))+\\
&+\sum_{r\notin\{j-1,j\}}C(x_{j+1},\ldots,Q(x_r,x_{r+1}),\ldots,x_{k+1},y,x_1,\ldots,x_{j-1})\Big]\\
&=h_j\left([Q,C](x_{j+1},\ldots,x_{k+1},y,x_1,\ldots,x_{j-1})\right)=\left([Q,C]_j(x_1,\ldots,h_j,\ldots,x_{k+1})\right)(y).
\endaligned
$$

\

\noindent
\underline{Case: $j=k+1$}, then:
$$\aligned
D_{k+1}&=Q(x_1,C_k(x_2,\ldots,h_{k+1}))+Q(C(x_1,\ldots,x_k),h_{k+1})+C_k(x_1,\ldots,Q(x_k,h_{k+1}))+\\
&+\sum_{r=1}^{k-1}C_k(x_1,\ldots,Q(x_r,x_{r+1}),\ldots,h_{k+1}).
\endaligned
$$
The value of this form on $y\in V[1]$ is:
$$\aligned
D_{k+1}(y)&=h_{k+1}\Big[C(Q(y,x_1),x_2,\ldots,x_k)+Q(y,C(x_1,\ldots,x_k))+Q(C(y,x_1,\ldots,x_{k-1}),x_k)+\\
&\hskip 1cm+\sum_{r=1}^{k-1}C(y,x_1,\ldots,Q(x_r,x_{r+1}),\ldots,x_k)\Big]\\
&=h_1\left([Q,C](x_2,\ldots,x_{k+1},y)\right)=\left([Q,C]_{k+1}(x_1,\ldots,h_{k+1})\right)(y).
\endaligned
$$

\

Finally,
$$
[\tilde{Q},\tilde{C}]=\widetilde{[Q,C]}=\widetilde{d_Hc[1]},
$$
where $d_H$ is the Hochschild coboundary operator on the bimodule $M$. If $d_P$ is the Pinczon coboundary operator, defined by:
$$
\{\Omega_{\tilde{Q}},\Omega_{\tilde{C}}\}\tilde{~}=\Omega_{d_P\tilde{C}},
$$
this can be written $d_P\tilde{C}=\widetilde{d_Hc[1]}$.\\

\begin{prop}

\

Let $(V,q)$ be an associative algebra, and $\Phi:c\mapsto\tilde{C}$ the map associating to any multilinear mapping $c$ from $V^k$ into $M$, with degree $2-k$, the $\tilde{B}$-quadratic application $\tilde{C}$. Then $\Phi$ is a morphism from the Hochschild cohomology complex for the $(V,q)$ bimodule $M$ and the Pinczon cohomology complex of cyclic forms $\mathcal C(\tilde{V})$ on $\tilde{V}$.\\
\end{prop}


\section{Commutative Pinczon algebras}


\

\subsection{Commutative quadratic algebras}

\

Consider a quadratic associative algebra $(V,b,q)$, but suppose now $q$ is commutative, that means, without any sign rule:
$$
q(x_1,x_2)=(-1)^{|x_1||x_2|}q(x_2,x_1),
$$
and consider, as above, the corresponding coderivation $Q$. It is now anticommutative, with degree 1, and seen as a map from the quotient $\underline\otimes^2V[1]$ of $\otimes^2V[1]$ by the image, into $V[1]$, of the $1,1$ shuffle product (with our sign rule):
$$
sh_{1,1}(x_1,x_2)=x_1\otimes x_2+x_2\otimes x_1.
$$

Recall that a $p,q$ shuffle $\sigma$ is a permutation $\sigma\in\mathfrak S_{p+q}$ such that $\sigma(1)<\ldots<\sigma(p)$ and $\sigma(p+1)<\ldots<\sigma(p+q)$. Denote $Sh(p,q)$ the set of all such shuffles. Then the $p,q$ shuffle product on $\otimes^+V[1]$ is (with the Koszul rule)
$$
sh_{p,q}(x_{[1,p]},x_{[p+1,p+q]})=\sum_{\sigma\in Sh(p,q)}x_{\sigma^{-1}(1)}\otimes\ldots\otimes x_{\sigma^{-1}(p+q)}.
$$

Denote $\underline{\otimes}^n(V[1])$ the quotient of $\otimes^nV[1]$ by the sum of all the image of the maps $sh_{p,n-p}$ ($0<p<n$). Similarly, abusively note $\underline{x}_{[1,n]}=x_1\underline{\otimes}\ldots\underline{\otimes}x_n$ the class of $x_1\otimes\ldots\otimes x_n$, where the $x_i$ belong to $V[1]$. Finally, $\underline{\otimes}^+(V[1])$ is the sum of all $\underline{\otimes}^n(V[1])$ ($n>0$).\\

After replacing $\otimes^+V[1]$ by $\underline{\otimes}^+V[1]$, let us replace the comultiplication $\Delta$ by a Lie cocrochet $\delta$, defined by
$$
\delta(\underline{x}_{[1,n]})=\sum_{j=1}^{n-1}\underline{x}_{[1,j]}\bigotimes \underline{x}_{[j+1,n]}-\underline{x}_{[j+1,n]}\bigotimes \underline{x}_{[1,j]}.
$$
In fact $\delta$ is well defined on the quotient and any coderivation $Q$ of $\delta$ is characterized by its Taylor expansion $Q=\sum_k Q_k$ where each $Q_k$ is a linear map from $\underline{\otimes}^kV[1]$ into $V[1]$ (\cite{AAC2,BGHHW}).\\

\begin{defn}

\

A structure of $C_\infty$ algebra, or commutative algebra up to homotopy, on $V$ is a degree 1 coderivation $Q$ of $\delta$, on $\underline{\otimes}^+V[1]$, such that $[Q,Q]=0$.\\
\end{defn}

Such mappings $Q$ are vanishing on the image of shuffle products. Consider now cyclic forms $\Omega_Q$ such that $Q$ is vanishing on shuffle products, or:

\begin{defn} 

\

A $(k+1)$-linear cyclic form $\Omega$ on the vector space $V[1]$ is vanishing on shuffle products if and only if, for any $y$, 
$(x_1,\ldots,x_k)\mapsto \Omega(x_1,\ldots,x_k,y)$ is vanishing on shuffle products.

Denote $\mathcal C_{vsp}(V)$ the space of cyclic, vanishing on shuffle products multilinear forms on $V[1]$.\\
\end{defn}

\begin{prop}

\

Suppose $\{~,~\}$ is a Pinczon bracket on the space $\mathcal C(V)$ of cyclic multilinear forms on $V[1]$. Then $\mathcal C_{vsp}(V)$ is a Lie subalgebra of $(\mathcal C(V),\{~,~\})$.\\
\end{prop}

\begin{proof}

In fact, the Pinczon bracket defines a non degenerate form $b$ on $V$, thus $B$ on $V[1]$, any form $\Omega$ can be written as $\Omega=\Omega_Q$, with
$$
\Omega_Q(x_1,\ldots,x_{k+1})=B(Q(x_1,\ldots,x_k),x_{k+1}).
$$

Therefore $\Omega$ is in $\mathcal C_{vsp}(V)$ if and only if $Q$ is vanishing on shuffle products. Now, it is known that if $Q$, $Q'$ are vanishing on shuffle products, then $[Q,Q']$ is also vanishing on shuffle products (see for instance \cite{AAC2}). This proves the proposition.\\ 
\end{proof}

\begin{defn}

\

A commutative Pinczon algebra $(\mathcal C_{vsp}(V),\{~,~\},\Omega)$ is a vector space $V$, such that $\mathcal C(V)$ is equipped with a Pinczon bracket $\{~,~\}$, and an element $\Omega$ in $\mathcal C_{vsp}(V)$, with degree 3, such that $\{\Omega,\Omega\}=0$.\\

\end{defn}

If $\Omega$ is trilinear, then a commutative Pinczon algebra is simply a quadratic commutative algebra $(V,q,b)$, where $b$ is the symmetric non degenerated form coming from the restriction of the Pinczon bracket to $\left((V[1])^\star\right)^2$, and $q$ is the bilinear mapping associated to $Q$ such that $\Omega=\Omega_Q$.\\

\begin{prop}

\

Let $(\mathcal C_{vsp}(V),\{~,~\},\Omega)$ be a commutative Pinczon algebra, then there is an unique symplectic form $B$ on $V[1]$ and an unique vanishing on shuffle products, $B$-quadratic coderivation $Q$ of $(\underline{\bigotimes}^+(V[1]),\delta)$ such that $\Omega=\Omega_Q$.

Conversely, if $(V,b)$ is a vector space equipped with a non degenerated symmetric bilinear form $b$, any $B$-quadratic structure $Q$ of $C_\infty$ algebra defines an unique commutative Pinczon algebra $(\mathcal C_{vsp}(V),\{~,~\},\Omega_Q)$.\\
\end{prop}

In other words, the commutative Pinczon algebra $(\mathcal C_{vsp}(V),\{~,~\},\Omega)$ is a structure of quadratic commutative algebra up to homotopy on $V$.\\

\subsection{(Bi)modules and Harrison cohomology}

\

Any module $M$ on a commutative algebra $(V,q)$ can be viewed as a natural bimodule, the right action coinciding with the left one.\\

Let now $(V,q)$ be a commutative algebra, and $M$ a $(V,q)$-bimodule. Let us repeat exactly the preceding construction of the semidirect product $W=V\rtimes M$, with the associative and commutative multiplication:
$$
q_W((x,a),(y,b))=(xy,x\cdot b+a\cdot x).
$$
Consider then the quadratic associative algebra $(\tilde{V},\tilde{b},\tilde{q})$, where $\tilde{V}=W\times W^\star$, $\tilde{q}$ is
$$
\tilde{q}((x,a,g,h),(x',a',g',h'))=(xx',xa'+ax',x\cdot g'+g\cdot x'+a\cdot h'+ h\cdot a',x\cdot h'+h\cdot x'),
$$
and $\tilde b$,
$$
\tilde{b}((x,a,g,h),(x',a',g',h'))=g(x')+h(a')+g'(x)+h'(a).
$$

By construction, $(\tilde{V},\tilde{b},\tilde{q})$ is a quadratic commutative algebra, and:\\

\begin{defn}

\

Let $(V,q)$ be a commutative algebra and $M$ a bimodule. The double semi-direct product of $V$ by $M$ is the quadratic commutative algebra $(\tilde{V},\tilde{b},\tilde{q})$.

The corresponding commutative Pinczon algebra $(\mathcal C_{vsp}(\tilde{V}),\{~,~\}\tilde{~},\Omega_{\tilde{Q}})$ is the Pinczon double semi-direct product of $V$ by $M$.\\
\end{defn}

As above, look now for a $k$-linear mapping $c$ from $V^k$ into $M$, with degree $2-k$ and vanishing on shuffle products, and consider the corresponding map $\tilde{C}$, with degree 1, from $\tilde{V}[1]^k$ into $\tilde{V}[1]$,
$$
\tilde{C}((x_1,a_1,g_1,h_1),\ldots,(x_k,a_k,g_k,h_k))=(0,C(x_1,\ldots,x_k),\sum_{j=1}^kC_j(x_1,\ldots,h_j,\ldots,x_k),0).
$$
We saw that:
$$
[\tilde{Q},\tilde{C}]=\widetilde{[Q,C]}=\widetilde{d_Hc[1]},
$$
where $d_H$ is the Hochschild coboundary operator on the bimodule $M$. But if we restrict ourselves to $(V[1])^{k+1}$, this is:
$$
[\tilde{Q},\tilde{C}]((x_1,0,0,0),\ldots,(x_{k+1},0,0,0))=d_Hc[1](x_1,\ldots,x_{k+1})=d_{Ha}c[1](x_1,\ldots,x_{k+1}),
$$
Therefore, if $d_P$ is the Pinczon coboundary operator:
$$
\{\Omega_{\tilde{Q}},\Omega_{\tilde{C}}\}\tilde{~}=\Omega_{d_P\tilde{C}},
$$
this can be written $d_P\tilde{C}=\widetilde{d_{Ha}c[1]}$.\\

\begin{prop}

\

Let $(V,q)$ be a commutative algebra, and $\Phi:c\mapsto\tilde{C}$ the map associating to any multilinear mapping $c$ from $V^k$ into $M$, with degree $2-k$ and vanishing on skew shuffle products, the vanishing on shuffle products $\tilde{B}$-quadratic application $\tilde{C}$. Then $\Phi$ is a morphism from the Harrison cohomology complex for the $(V,q)$ bimodule $M$ and the Pinczon cohomology complex of cyclic, vanishing on shuffle products forms $\mathcal C_{vsp}(\tilde{V})$ on $\tilde{V}$.\\
\end{prop}


\section{Pinczon Lie algebras}


\

\subsection{Quadratic Lie algebras}

\

In this section, we consider the Lie algebra structure. Suppose $(V,q)$ is a (graded) Lie algebra, with $|q|=0$,  then the corresponding Bar resolution consists in replacing the space $\otimes^+V[1]$ of tensors by the subspace $S^+(V[1])$ of totally symmetric tensors, spanned by the symmetric products ($k>0$):
$$
x_1\cdot\ldots\cdot x_k=\sum_{\sigma\in\mathfrak S_k}x_{\sigma^{-1}(1)}\otimes\ldots\otimes x_{\sigma^{-1}(k)}.
$$
Then, consider the natural comultiplication $\Delta$ on $S^+(V[1])$:
$$
\Delta(x_1\cdot\ldots\cdot x_k)=\sum_{\begin{smallmatrix}I\sqcup J=[1,k]\\
0<\#I<k\end{smallmatrix}} x_{\cdot I}\bigotimes x_{\cdot J},
$$
where, if the subset $I$ is $I=\{i_1<\ldots<i_r\}$, $x_{\cdot I}$ means $x_{i_1}\cdot\ldots\cdot x_{i_r}$.

As above, any coderivation $Q$ of the comultiplication $\Delta$ is characterized by its Taylor coefficients $Q_k$, totally symmetric $k$-linear maps from $V[1]^k$ into $V[1]$, these maps are seen as linear maps from $S^k(V[1])$ into $V[1]$.

The bracket of two such coderivations $Q$, $Q'$ becomes:
$$
[Q,Q'](x_1\cdot\ldots\cdot x_{k+k'-1})=\sum_{\begin{smallmatrix}I\sqcup J=[1,k+k'-1]\\
\#J=k'\end{smallmatrix}} Q(Q'(x_{\cdot J})\cdot x_{\cdot I})-\sum_{\begin{smallmatrix}I\sqcup J=[1,k+k'-1]\\
\#I=k\end{smallmatrix}}Q'(Q(x_{\cdot I})\cdot x_{\cdot J}),
$$

Now to the Lie structure $q$ on $V$, there corresponds a totally symmetric bilinear map $Q:S^2(V[1])\longrightarrow V[1]$, thus a degree 1 coderivation, still denoted $Q$, of $\Delta$ and the Jacobi identity for $q$ is equivalent to the relation $[Q,Q]=0$. The corresponding cohomology is the Chevalley cohomology.\\

For any $k$, consider $S^k(V[1])$ as a subspace of $\otimes^k V[1]$, with a projection (up to a numerical factor) $Sym$ from $\otimes^k V[1]$ onto $S^k(V[1])$:
$$
Sym(x_1\otimes\ldots\otimes x_k)=x_1\cdot\ldots\cdot x_k.
$$

Consider now the space $L^k(V[1])$ of $k$-linear maps $Q$ from $V[1]^k$ into $V[1]$. By restriction to $S^k(V[1])$, $Q$ defines a totally symmetric multilinear map:
$$
Q^{Sym}(x_1\cdot\ldots\cdot x_k)=Q\circ Sym(x_1\otimes\ldots\otimes x_k).
$$
Therefore, the space $L(S^k(V[1]),V[1])$ of totally symmetric $k$-linear maps is a quotient of $L^k(V[1])$.\\

Similarly, if $\mathcal C^{k+1}$ is the space of $(k+1)$-linear cyclic forms $\Omega$, by restriction to $S^{k+1}(V[1])$, $\Omega$ defines a totally symmetric multilinear form, denote it $\Omega^{Sym}$:
$$
\Omega^{Sym}(x_1\cdot\ldots\cdot x_{k+1})=\Omega\circ Sym(x_1\otimes\ldots\otimes x_{k+1}).
$$
As above, the space $\mathcal C(V)|_S$ of the restrictions of cyclic multilinear forms to $S^+(V[1])$ is a quotient of $\mathcal C(V)$.\\

Suppose there is a Pinczon bracket $\{~,~\}$ on $\mathcal C(V)$, then any $\Omega$ can be written $\Omega=\Omega_Q$, with $Q$ $k$-linear and $B$-quadratic. But, any element $\sigma$ in $\mathfrak S_{k+1}$ can be written in an unique way as a product $\tau\circ\rho$ where $\tau$ is in $\mathfrak S_k$, viewed as a subgroup of $\mathfrak S_{k+1}$ and $\rho$ in $Cycl$. Therefore, with our notation,
$$
\Omega^{Sym}=(\Omega_Q)^{Sym}=(\Omega_{Q^{Sym}})^{Cycl}=(k+1)\Omega_{Q^{Sym}}.
$$

\begin{prop}

\

The bracket defined on $L(\otimes^+V[1],V[1])$ by the commutator of coderivations in $(\otimes^+V[1],\Delta)$, induces a well defined bracket on the quotient $L(S^+(V[1]),V[1])$, this bracket is the above commutator of coderivations in $(S^+(V[1]),\Delta)$:
$$
\left[Q^{Sym},(Q')^{Sym}\right]=[Q,Q']^{Sym}.
$$

Any Pinczon braket $\{~,~\}$ on the space $\mathcal C(V)$ of multilinear cyclic forms on $V[1]$ induces a well defined bracket on the quotient $\mathcal C(V)|_S$, this bracket denoted $\{~,~\}|$ is:
$$
\left\{\Omega^{Sym},(\Omega')^{Sym}\right\}|=\left(\{\Omega,\Omega'\}\right)^{Sym}=\frac{k+k'}{(k+1)!(k'+1)!} \sum_i\iota_{e_i}\left(\Omega^{Sym}\right)\cdot\iota_{e'_i}\left((\Omega')^{Sym}\right).
$$
\end{prop}

\begin{proof}

The first assertion is a simple computation. Put
$$\aligned
(Q\circ Q')(x_1,\ldots,x_{k+k'-1})&=\sum_{r=1}^{k-1}Q(x_1,\ldots,x_{r-1},Q'(x_r,\ldots,x_{r+k'-1}),x_{r+k'},\ldots,x_{k+k'-1})\\
(Q\circ Q')(x_{[1,k+k'-1]})&=\sum_{r=1}^{k-1}Q(x_{[1,r-1]}\otimes Q'(x_{[r,r+k'-1]})\otimes x_{[r+k',k+k'-1]}).
\endaligned
$$

Then, denoting abusively $x_{\sigma^{-1}([a,b])}$ the quantity $x_{\sigma^{-1}(a)}\otimes\ldots\otimes x_{\sigma^{-1}(b)}$,
$$
(Q\circ Q')^{Sym}(x_{\cdot[1,k+k'-1]})=\sum_{r=1}^{k-1}\sum_{\sigma\in\mathfrak S_{k+k'-1}}Q(x_{\sigma^{-1}([1,r-1])},Q'(x_{\sigma^{-1}([r,r+k'-1])}),x_{\sigma^{-1}([r+k',k+k'-1])})\\
$$

For any subset $J$ in $[1,k+k'-1]$, with $\#J=k'$, for any $r$, put $I=[1,k+k'-1]\setminus J$ and, for any term in the above sum for which $\sigma^{-1}([r,r+k'-1])=J$, define two permutaions $\tau\in\mathfrak S_{k'}$, and $\rho\in\mathfrak S_k$ as follows:
\begin{itemize}
\item[] write $J=\{j_1<\ldots<j_{k'}\}$, then put
$$
j_{\tau^{-1}(t)}=\sigma^{-1}(r+t-1)\quad (1\leq t\leq k'),
$$
\item[] write $\{0\}\cup I=\{i_1<\ldots<i_k\}$, then put
$$
i_{\rho^{-1}(t)}=\sigma^{-1}(t)\quad (1\leq t\leq r-1),\quad i_{\rho^{-1}(r)}=0, \quad i_{\rho^{-1}(t)}=\sigma^{-1}(t-1)\quad (r+1\leq t\leq k).
$$
\end{itemize} 
The correspondence $(r,\sigma)\mapsto(r,\tau,\rho)$ is one-to-one and, suming up, we get:
$$
(Q\circ Q')^{Sym}(x_{\cdot[1,k+k'-1]})=\sum_{\begin{smallmatrix}I\sqcup J=[1,k+k'-1]\\
\#J=k'\end{smallmatrix}}Q^{Sym}((Q')^{Sym}(x_{\cdot J})\cdot x_{\cdot I}).
$$

Therefore, the quotient bracket is well defined and is the bracket of coderivation of $(S^+(V[1]),\Delta)$:
$$
[Q,Q']^{Sym}=(Q\circ Q')^{Sym}-(Q'\circ Q)^{Sym}=\left[Q^{Sym},(Q')^{Sym}\right].
$$

Let now $\{~,~\}$ be a Pinczon bracket on the space $\mathcal C(V)$ of cyclic multilinear forms on $V[1]$. Thus there is a symplectic form $B$ on $V[1]$ and any cyclic form $\Omega$ is written as $\Omega_Q$. Then:
$$
\{\Omega,\Omega'\}^{Sym}=\{\Omega_Q,\Omega_{Q'}\}^{Sym}=\Omega_{[Q,Q']}^{Sym}=(k+k')\Omega_{[Q,Q']^{Sym}}=(k+k')\Omega_{\left[Q^{Sym},(Q')^{Sym}\right]}.
$$
The last bracket is the commutator of coderivations in $S^+(V[1])$, thus
$$\aligned
\{\Omega,\Omega'\}^{Sym}(x_{\cdot[1,k+k']})&=(k+k')\sum_{\begin{smallmatrix}I\sqcup J=[1,k+k'-1]\\ \#J=k'\end{smallmatrix}}B(Q^{Sym}((Q')^{Sym}(x_{\cdot J})\cdot x_{\cdot I}),x_{k+k'})-\\
&\hskip 1cm-(k+k')\sum_{\begin{smallmatrix}I\sqcup J=[1,k+k'-1]\\ \#I=k\end{smallmatrix}}B((Q')^{Sym}(Q^{Sym}(x_{\cdot I})\cdot x_{\cdot J}),x_{k+k'})\\
&=(k+k')\sum_{\begin{smallmatrix}I\sqcup J=[1,k+k']\\ \#J=k'\end{smallmatrix}}B(Q^{Sym}(x_{\cdot I}),(Q')^{Sym}(x_{\cdot J})).
\endaligned
$$

On the other hand,
$$
\iota_{e_i}\left(\Omega_Q^{Sym}\right)(x_{\cdot[1,k]})=(k+1)B(Q^{Sym}(x_{\cdot[1,k]}),e_i)=(k+1)\left(\iota_{e_i}\Omega\right)^{Sym}(x_{\cdot[1,k]}).
$$
Therefore
$$\aligned
\sum_i\iota_{e_i}\left(\Omega_Q^{Sym}\right)\otimes\iota_{e'_i}\left(\Omega_{Q'}^{Sym}\right)(x_{\cdot[1,k+k']})&=(k+1)(k'+1)B(Q^{Sym}(x_{\cdot[1,k]}),(Q')^{Sym}(x_{\cdot[k,k+k']})),\\
\endaligned
$$
and
$$\aligned
\sum_i\iota_{e_i}\left(\Omega_Q^{Sym}\right)\cdot\iota_{e'_i}\left(\Omega_{Q'}^{Sym}\right)(x_{\cdot[1,k+k']})&=\sum_i\left(\iota_{e_i}\left(\Omega_Q^{Sym}\right)\otimes\iota_{e'_i}\left(\Omega_{Q'}^{Sym}\right)\right)^{Sym}(x_{\cdot[1,k+k']})\\
&\hskip-1cm=(k+1)!(k'+1)!\sum_{\begin{smallmatrix}I\sqcup J=[1,k+k']\\ \#J=k'\end{smallmatrix}}B(Q^{Sym}(x_{\cdot I}),(Q')^{Sym}(x_{\cdot J}))\\
\endaligned
$$

This proves that the Pinczon bracket induces a well defined bracket on the quotient $\mathcal C(V)|_S$, and gives the last assertion.\\
\end{proof}

Explicitely, if the forms $\Omega$ and $\Omega'$ are totally symmetric ($\Omega$, $\Omega'$ belong to $\mathcal C(V)|_S$), $\Omega^{Sym}=(k+1)!\Omega$ and the Pinczon bracket on the quotient becomes:
$$
\{\Omega,\Omega'\}|=(k+k')\sum_i\iota_{e_i}\Omega\cdot\iota_{e'_i}\Omega'.
$$
We refind here the expression given in \cite{PU,DPU}.\\

\subsection{Quadratic $L_\infty$ algebras}

\

We are now in position to define the Pinczon Lie algebras:\\

\begin{defn}

\

A Pinczon Lie algebra $(\mathcal C(V)|_S,\{~,~\}|,\Omega)$ is a vector space $V$, such that $\mathcal C(V)$ is equipped with a Pinczon bracket $\{~,~\}$, and a symmetric multilinear form $\Omega$, with degree 3, such that $\{\Omega,\Omega\}|=0$.\\
\end{defn}

Recall that a structure of $L_\infty$ algebra (or Lie algebra up to homotopy) on $V$ is a degree 1 coderivation $Q$ of $(S^+(V[1]),\Delta)$, such that the commutator $[Q,Q]$ vanishes. With the preceding computations, a Pinczon Lie algebra is in fact a quadratic $L_\infty$ algebra.\\

\begin{prop}

\

Let $(\mathcal C(V)|_S,\{~,~\}|,\Omega)$ be a Pinczon Lie algebra, then there is an unique symplectic form $B$ on $V[1]$ and an unique $B$-quadratic coderivation $Q$ of $(S^+(V[1]),\Delta)$ such that $\Omega=\Omega_Q$.

Conversely, if $(V,b)$ is a vector space equipped with a non degenerated symmetric bilinear form $b$, any $B$-quadratic structure $Q$ of $L_\infty$ algebra defines an unique Pinczon Lie algebra $(\mathcal C(V)|_S,\{~,~\}|,\Omega_Q)$.\\
\end{prop}

\subsection{Modules and Chevalley cohomology}

\

Suppose in this section that $(V,q)$ is a Lie algebra ($q(x,y)=[x,y]$), and $M$ a finite dimensional $(V,q)$-module. To refind the corresponding Chevalley coboundary operator $d_{Ch}$, build, as above, the double semidirect product of $V$ by $M$.\\

First, consider the semidirect product $W=V\rtimes M$, that is the vector space $V\times M$, equipped with the Lie bracket:
$$
q_W((x,a),(y,b))=([x,y],x\cdot b-y\cdot a),
$$
since $W$ is a $(W,q_W)$ bimodule, its dual $W^\star$ is also a $(W,q_W)$-module, with:
$$
\big((x,a)\cdot f\big)(z,c)=f(q_W((z,c),(x,a))).
$$

With these objects, define a strucutre of Lie algebra on the space $\tilde{V}=W\times W^\star$, by putting:
$$
\tilde{q}((x,a,g,h),(x',a',g',h'))=([x,x'],x\cdot a'-x'\cdot a,x\cdot g'-x'\cdot g+a\cdot h'-a'\cdot h,x\cdot h'-x'\cdot h),
$$
and a non degenerated symmetric bilinear form $\tilde b$, by
$$
\tilde{b}((x,a,g,h),(x',a',g',h'))=g(x')+h(a')+g'(x)+h'(a).
$$
A direct computation (see \cite{BB,MR}) shows that $(\tilde{V},\tilde{b},\tilde{q})$ is a quadratic Lie algebra,\\

\begin{defn}

\

Let $(V,q)$ be a Lie algebra and $M$ a module. The double semi-direct product of $V$ by $M$ is the quadratic Lie algebra $(\tilde{V},\tilde{b},\tilde{q})$.

The corresponding Pinczon Lie algebra $(\mathcal C(\tilde{V})|_S,\{~,~\}\tilde{~}|,\Omega_{\tilde{Q}})$ is the Pinczon double semi-direct product of $V$ by $M$.\\
\end{defn}

Look now for a $k$-linear mapping $c$ from $V^k$ into $M$, with degree $|c|=2-k$, and totaly skewsymmetric. As above, associate to it the mapping:
$$
\tilde{C}((x_1,a_1,g_1,h_1),\ldots,(x_k,a_k,f_k,g_k))=(0,C(x_1,\ldots,x_k),\sum_{j=1}^kC_j(x_1,\ldots,h_j,\ldots,x_k),0),
$$
with
$$
C_j(x_1,\ldots,h_j,\ldots,x_k)(y)=h_j(C(x_{j+1},\ldots,x_k,y,x_1,\ldots,x_{j-1})).
$$
Clearly, $\tilde{C}$ is totaly symmetric from $\tilde{V}[1]^k$ into $\tilde{V}[1]$. More precisely,
$$\aligned
\widetilde{C^{Sym}}&((x_1,a_1,g_1,h_1),\ldots,(x_k,a_k,g_k,h_k))=\\
&=\sum_{\sigma\in\mathfrak{S}_k}C(x_{\sigma^{-1}(1)},\ldots,x_{\sigma^{-1}(k)})+\sum_{j=1}^kh_{\sigma^{-1}(j)}C(x_{\sigma^{-1}(1)},\ldots,\cdot,\ldots,x_{\sigma^{-1}(k)})\\
&={\tilde{C}}^{Sym}((x_1,a_1,g_1,h_1),\ldots,(x_k,a_k,g_k,h_k)).
\endaligned
$$

Then, if $\tilde{Q}$ is associated to $\tilde{q}$ as above,
$$
[\tilde{Q},\tilde{C}]^{Sym}={\widetilde{[Q,C]}}^{Sym}=\widetilde{[Q,C]^{Sym}}=\widetilde{d_{Ch}c[1]}.
$$
Now, if $d_P$ is the Pinczon coboundary operator, defined by:
$$
\{\Omega_{\tilde{Q}},\Omega_{\tilde{C}}\}|=\Omega_{d_P\tilde{C}},
$$
this can be written $d_P\tilde{C}=(2+k)\widetilde{d_{Ch}c[1]}$.\\

\begin{prop}

\

Let $(V,q)$ be a Lie algebra, and $\Phi:c\mapsto\tilde{C}$ the map associating to any skewsymmetrix multilinear mapping $c$ from $V^k$ into $M$, with degree $2-k$, the symmetric $\tilde{B}$-quadratic application $\tilde{C}$. Then $\Phi$ is a morphism from the Chevalley cohomology complex for the $(V,q)$ bimodule $M$ and the Pinczon cohomology complex of symmetric forms $\mathcal C(\tilde{V})|_S$ on $\tilde{V}$.\\
\end{prop}



\end{document}